\documentclass[a4paper,12pt]{article}
\usepackage{amsmath}
\usepackage{setspace}
\usepackage{titlesec}
\usepackage{amssymb}
\usepackage{amsthm}
\usepackage{enumerate}
\usepackage{fancyhdr}
\usepackage{mathtools}
\usepackage[utf8]{inputenc}
\usepackage[all]{xy}
\usepackage{framed}
\usepackage[mathscr]{euscript}

\newtheorem{theorem}{Theorem}

\newtheorem{corollary}[theorem]{Corollary}
\newtheorem{proposition}[theorem]{Proposition}
\newtheorem{lemma}[theorem]{Lemma}
\newcommand{\eps}{\mathbin{\e}}
\newcommand{\e}{\varepsilon}

\newcommand{\catMod}[1]{\mathbf{#1{-}Mod}}	

\newcommand{\catLCS}{\mathbf{LCS}}	
\newcommand{\catLCMod}[1]{\mathbf{#1{-}LCMod}}	
\newcommand{\catLCModt}[2]{\mathbf{#1{-}LCMod{-}#2}}	
\newcommand{\catVB}{\mathbf{VB}}	

\newcommand{\cL}{\mathcal{L}}
\newcommand{\bD}{\mathbb{D}}
\newcommand{\bK}{\mathbb{K}}
\newcommand{\bR}{\mathbb{R}}
\newcommand{\bC}{\mathbb{C}}
\newcommand{\bZ}{\mathbb{Z}}
\newcommand{\cH}{\mathcal{H}}
\newcommand{\cD}{\mathcal{D}}
\newcommand{\cK}{\mathcal{K}}
\newcommand{\cM}{\mathcal{M}}
\newcommand{\sA}{\mathscr{A}}
\newcommand{\sB}{\mathscr{B}}
\newcommand{\sC}{\mathscr{C}}
\DeclareMathOperator{\id}{id}
\DeclareMathOperator{\Vol}{Vol}
\DeclareMathOperator{\acx}{acx}
\DeclareMathOperator{\cacx}{\overline{\acx}}
\newcommand{\coleq}{\coloneqq}

\author{E.~A.~Nigsch\footnote{Wolfgang Pauli Institute, Oskar-Morgenstern-Platz 1, 1090 Vienna, Austria. E-Mail: eduard.nigsch@univie.ac.at, Phone: +43 4277 50760}}
\title{On regularization of vector distributions\\ on manifolds}

\titleformat{\section}
  {\normalfont\large\scshape}{\thesection}{1em}{}
\begin{document}

\maketitle

\textbf{Abstract. }One can represent Schwartz distributions with values in a vector bundle $E$ by smooth sections of $E$ with distributional coefficients. Moreover, any linear continuous operator which maps $E$-valued distributions to smooth sections of another vector bundle $F$ can be represented by sections of the external tensor product $E^* \boxtimes F$ with coefficients in the space $\cL(\cD', C^\infty)$ of operators from scalar distributions to scalar smooth functions. We establish these isomorphisms topologically, i.e., in the category of locally convex modules, using category theoretic formalism in conjunction with L.~Schwartz' notion of $\eps$-product.



\section{Introduction and Preliminaries}

Our aim is to show, given any vector bundles $E \to M$ and $F \to N$, the isomorphisms
\begin{gather}
\cD'(M, E) \cong \Gamma(M, E) \otimes_{C^\infty(M)} \cD'(M) \cong \cL_{C^\infty(M)} \bigl( \Gamma(M, E^*), \cD'(M) \bigr) \label{iso1} \\
\begin{aligned}
\cL \bigl( \cD'(M, E), \Gamma(N, F) \bigr) & \cong \Gamma ( M \times N, E^* \boxtimes F ) \otimes_{C^\infty(M \times N)} \cL \bigl( \cD'(M), C^\infty(N) \bigr) \\
&\cong \cL_{C^\infty(M \times N)} \bigl( \Gamma(M \times N, E \boxtimes F^*), \cL ( \cD'(M), C^\infty(N) ) \bigr)
\end{aligned}
 \label{iso2}
\end{gather}
in the category of locally convex modules (see Section \ref{sec_preliminaries} for notation). 

\eqref{iso1} is fundamental for the the extension of L.~Schwartz' theory of distributions to the case of distributions on manifolds with values in vector bundles. In fact, it enables one to view distributional sections as smooth sections with distributional coefficients, and hence allows their description by local coordinates. Naturally, it is desirable to establish that this is a topological isomorphism, for instance in order to obtain convergence of a sequence in $\cD'(M,E)$ from convergence of its coordinates. 

The motivation to consider \eqref{iso2} comes from the field of nonlinear generalized functions (or Colombeau algebras). Such algebras, containing distributions as a vector subspace and smooth functions as a faithful subalgebra (whilst having optimal properties in light of the Schwartz impossibility result about multiplication of distributions \cite{Schwartz}), are commonly constructed by representing Schwartz distributions by families of smooth functions, which amounts to regularizing them in a particular way (cf.~\cite{papernew}). Although this is straightforward in the scalar case, the construction of a (diffeomorphism invariant) algebra of generalized tensor fields is considerably more complicated (cf.~\cite{global,global2}). The construction in \cite{global2} involves the ingredients $\Gamma(M \times M, E^* \boxtimes E)$ and $\cL(\cD'(M), C^\infty(M))$ (the latter albeit only in disguise) in order to regularize distributions in $\cD'(M,E)$ in a coordinate-independent manner, but it was not exploited there that any linear continuous mapping $\cD'(M, E) \to \Gamma(M, E)$ necessarily is of the form exhibited by \eqref{iso2}. The spaces of so-called smoothing operators $\cL(\cD'(M), C^\infty(M))$ (scalar case) and $\cL(\cD'(M, E), \Gamma(M, E))$ (vector valued case) were found to be the optimal starting points for a general, geometric construction of Colombeau algebras (\cite{papernew}; cf.~also \cite{found,global,gdnew}). Hence, it is desirable to obtain isomorphism \eqref{iso2} in a topological setting for two reasons: first, it allows to relate the construction of \cite{global2} to the new, more natural approach to Colombeau algebras given in \cite{papernew}; and second, it allows to split the regularization of vector valued distributions into a smooth vectorial part $\Gamma(M \times N, E^* \boxtimes F)$ and a regularizing part $\cL(\cD'(M), C^\infty(N))$. This splitting is expected to be of essential practical importance in the further development of spaces of nonlinear generalized sections applicable to problems of nonlinear distributional geometry (\cite{foundgeom,genpseudo,genrel}).

It is evident that the above isomorphisms are straightforward to obtain on the algebraic level by reduction to the trivial line bundles (see Section \ref{sec_reduction}). On the topological level, however, they require the proper handling of topologies on modules and their tensor products as well as related spaces of linear mappings. Here we draw on concepts of A.~Grothendieck and L.~Schwartz concerning topological tensor products and the theory of vector valued distributions (\cite{zbMATH03199982,FDVV,zbMATH03145498,zbMATH03145499}): first, we use the idea of endowing the tensor product $\cH \otimes \cK$ with a topology such that linear continuous mappings on it correspond to bilinear mappings on $\cH \times \cK$ which are hypocontinuous with respect to certain families of bounded sets; second, we employ the notion of $\eps$-product; and third, a key element of our proof (Lemma \ref{innsbruck} below) may be regarded as an application of L.~Schwartz' \textit{Th\'eor\`eme de croisement}, a cornerstone of his theory of vector valued distributions.

We remark that a version of \eqref{iso1} in the bornological setting (i.e., using the bornological tensor product) was obtained in \cite{sectop}.

\section{Preliminaries}\label{sec_preliminaries}

Let the field $\bK$ be fixed as $\bR$ or $\bC$ throughout. All locally convex spaces will be assumed to be Hausdorff and over $\bK$. For two locally convex spaces $\cH$ and $\cK$ we denote by $\cL(\cH,\cK)$ the space of continuous linear mappings from $\cH$ to $\cK$. Endowed with the topology of simple (or pointwise) convergence it will be denoted by $\cL_\sigma (\cH, \cK)$ and if it carries the topology of bounded convergence by $\cL_\beta(\cH, \cK)$ (\cite[Chapter III, \S 3, p.~81]{Schaefer}).

By $\cH \otimes_\lambda \cK$ for $\lambda \in \{ \beta, \iota \} $ we denote the algebraic tensor product $\cH \otimes \cK$ endowed with the finest locally convex topology such that the canonical mapping $\otimes \colon \cH \times \cK \to \cH \otimes \cK$ is $\lambda$-continuous, which means separately continuous in case $\lambda = \iota$ and hypocontinuous in case $\lambda = \beta$ (cf.~\cite[p.~10]{zbMATH03145499}); when we say hypocontinuous, if not specified otherwise we always mean this with respect to the families of bounded subsets of the respective spaces. There would be more possible choices for $\lambda$ but we will not need these here. Note that in both cases $\cH \otimes_\lambda \cK$ is Hausdorff.

$\cH \otimes_\lambda \cK$ is the unique locally convex space (up to isomorphism) with the following universal property: for each $\lambda$-continuous bilinear mapping from $\cH \times \cK$ into any locally convex space $\cM$ there exists a unique continuous mapping $\tilde f \colon \cH \otimes_\lambda \cK \to \cM$ such that $f = \tilde f \circ \otimes$. This correspondence defines a linear isomorphism between the vector spaces of all $\lambda$-continuous bilinear mappings $\cH \times \cK \to \cM$ and $\cL ( \cH \otimes_\lambda \cK, \cM)$ (\cite[p.~10]{zbMATH03145499}).

Let $\cH_i, \cK_i, \cM_i$ be locally convex spaces and $f_i \in \cL(\cH_i, \cK_i)$ for $i=1,2$. Then $f_1 \otimes f_2 \in \cL(\cH_1 \otimes_\lambda \cH_2, \cK_1 \otimes_\lambda \cK_2)$ for $\lambda \in \{ \beta, \iota \}$ (\cite[p.~14]{zbMATH03145499}) and for $g_i \in \cL(\cK_i, \cM_i)$ ($i=1,2$) we have $(g_1 \otimes g_2) \circ (f_1 \otimes f_2)= (g_1 \circ f_1 ) \otimes (g_2 \circ f_2)$, which turns $\otimes_\lambda$ into a functor $\catLCS \times \catLCS \to \catLCS$ (see below).

All manifolds will be assumed to be smooth, Hausdorff, second countable and finite dimensional. Given a manifold $M$ and a vector bundle $E \to M$ we denote by $C^\infty(M)$, $C^\infty_c(M)$, $\Gamma(M, E)$ and $\Gamma_c(M,E)$ the spaces of smooth functions $M \to \bK$, compactly supported smooth functions $M \to \bK$, smooth section of $E$ and compactly supported smooth sections of $E$, respectively. $C^\infty(M)$ and $\Gamma(M,E)$ carry their usual Fr\'echet topology (\cite[Chapter XVII, Section 2, p.~236]{zbMATH03423321}) and $\Gamma_c(M,E)$ the corresponding (LF)-topology. Writing $\Vol(M)$ for the volume bundle of $M$ (\cite[Chapter 16, p.~429]{Lee}) and $E^*$ for the dual bundle of $E$, the spaces of scalar and $E$-valued distributions on $M$ are defined as the dual spaces $\cD'(M) \coleq \Gamma_c\bigl(M, \Vol(M)\bigr)'$ and $\cD'(M, E) \coleq \Gamma_c\bigl(M, E^* \otimes \Vol(M)\bigr)'$, respectively, both endowed with the strong dual topology (\cite[Definition 3.1.4, p.~231]{GKOS}). Finally, $E \boxtimes F$ denotes the external tensor product of two vector bundles $E$ and $F$ (\cite[Chapter II, Problem 4, p.~84]{GHV}).

We will employ some notions from category theory, using \cite{zbMATH00626734,zbMATH05273392} for general background reference. Given a category $\sC$ and any two of its objects, $A$ and $B$, the set of morphisms from $A$ to $B$ will be denoted by $\sC (A, B)$. We will employ the following categories: $\catVB_M$, the category of smooth vector bundles over a fixed manifold $M$ with morphisms given by smooth vector bundle homomorphisms covering the identity mapping of $M$; $\catLCS$, the category of locally convex spaces with morphisms given by continuous linear mappings; $\catMod A$, the category of $A$-modules with morphisms given by $A$-linear mappings; and $\catLCMod A$, the category of locally convex $A$-modules with morphisms given by $A$-linear continuous mappings, as defined in Section \ref{sec_lcm}.

We will need certain functors to commute with coproducts. This will be obtained very easily in our setting because the categories and functors we are dealing with are \emph{additive}. We recall the relevant definitions from \cite{zbMATH05273392}:
a \emph{preadditive category} is a category $\sC$ together with an abelian group structure on each set $\sC(A, B)$ of morphisms such that the composition mappings $\sC (A, B) \times \sC (B, C) \to \sC (A, C)$, $(f,g) \mapsto g \circ f$ are group homomorphisms in each variable. A convenient feature of preadditive categories is that finite coproducts and finite products are the same objects (\cite[Proposition 1.2.4, p.~4]{zbMATH05273392}) and hence are called \emph{biproducts}. An \emph{additive category} then is a preadditive category with a zero object and such that all finite biproducts exist. A functor $F \colon \sA \to \sB$ between two preadditive categories is called \emph{additive} if for all objects $A,A'$ in $\sA$, the mapping
\[ F \colon \sA(A, A') \to \sB \bigl(F(A), F(A')\bigr),\ f \mapsto F(f) \]
is a group homomorphism. Most importantly, a functor is additive if and only if it preserves biproducts (\cite[Proposition 1.3.4, p.~9]{zbMATH05273392}).

Concerning our setting it is easy to see that the categories $\catVB_M$ and $\catLCS$ 
are additive; 
moreover, the functors $\textunderscore^* \colon \catVB_M \to \catVB_M$ (dual bundle), $\otimes \colon \catVB_M \times \catVB_M \to \catVB_M$ (tensor product of vector bundles) and $\boxtimes \colon \catVB_M \times \catVB_N \to \catVB_{M \times N}$ (external tensor product) as well as $\otimes_{\lambda} \colon \catLCS \times \catLCS \to \catLCS$ are additive. 
We omit the detailed proofs here because they amount to routine verification of well-known properties.

\section{Locally convex modules}\label{sec_lcm}

In this sections we are going to recall some needed definitions and properties of locally convex algebras as well as locally convex modules and their tensor products (see \cite{zbMATH03426281,MR1093462,MR857807} for additional information).

A \emph{locally convex algebra} $A$ is a locally convex space together with a separately continuous multiplication $A \times A \to A$ turning it into an associative commutative unitary algebra over $\bK$. Given a locally convex algebra $A$, a \emph{locally convex $A$-module} $\cH$ is a locally convex space which is an $A$-module such that module multiplication $A \times \cH \to \cH$ is separately continuous. For fixed $A$, the locally convex $A$-modules whose multiplication is $\lambda$-continuous (with $\lambda \in \{\iota, \beta\}$ as before) are the objects of an additive category $\catLCModt{A}{\lambda}$ whose morphisms are continuous $A$-linear mappings and whose biproducts are formed in $\catLCS$.
We simply write $\catLCMod{A}$ instead of $\catLCModt{A}{\iota}$.

Given a locally convex algebra $A$ and two locally convex $A$-modules $\cH$ and $\cK$ we denote by $\cL_A ( \cH, \cK )$ the space of all continuous $A$-linear mappings from $\cH$ to $\cK$. Endowed with the topology of simple or bounded convergence (i.e., the trace topology with respect to $\cL_\sigma(\cH, \cK)$ or $\cL_{\beta}(\cH, \cK)$, respectively) we denote it by by $\cL_{A, \sigma}(\cH,\cK)$ or $\cL_{A, \beta}(\cH,\cK)$, respectively. 

$\cL_{A, \sigma} (\cH, \cK)$ with its canonical $A$-module structure is a locally convex $A$-module.
Moreover, if $\cK$ has $\beta$-continuous multiplication then $\cL_{A, \beta}(\cH, \cK)$ has $\beta$-continuous multiplication as well. In fact, fix a $0$-neighbourhood $U_{B,V} = \{ \ell: \ell(B) \subseteq V \}$ in $\cL_{A, \beta} ( \cH, \cK)$ where $B \subseteq \cH$ is bounded and $V \subseteq \cK$ is a $0$-neighborhood. Given a bounded subset $A' \subseteq A$, choose a $0$-neighborhood $V' \subseteq \cK$ such that $A' \cdot  V' \subseteq V$; then $U_{B, V'}$ is a $0$-neighborhood in $\cL_{A, \beta}(\cH, \cK)$ such that $A' \cdot U_{B, V'} \subseteq U_{B, V}$. On the other hand, given a bounded subset $L \subseteq \cL_{A, \beta}(\cH, \cK)$ we can find a $0$-neighborhood $W \subseteq A$ such that $W \cdot L(B) \subseteq V$ (note that $L(B)$ is bounded); this means that $W \cdot L \subseteq U_{B,V}$, which proves the claim.

Moreover, $\cL_{A, \sigma} (-, -)$ is an additive 
functor $\catLCMod A \times \catLCMod A \to \catLCMod A$ which is contravariant in the first argument and covariant in the second argument. 
Similarly, $\cL_{A, \beta}$ is an additive functor $\catLCMod{A} \times \catLCModt{A}{\beta} \to \catLCModt{A}{\beta}$.

\begin{lemma}\label{blahlinprop}
Given a locally convex algebra $A$ and a locally convex $A$-module $\cH$, $\cH \cong \cL_{A, \sigma} (A, \cH)$ in $\catLCMod{A}$. If $\cH$ has hypocontinuous multiplication then $\cH \cong \cL_{A, \beta} (A, \cH)$ in $\catLCModt{A}{\beta}$.
\end{lemma}
\begin{proof}
 Algebraically, the isomorphism $\varphi \colon \cH \to \cL_{A} ( A, \cH )$ is given by $\varphi(x)(a) \coleq a \cdot x$ with inverse $\varphi^{-1}(l) \coleq l(1)$. Continuity of $\varphi$ and $\varphi^{-1}$ is clear in both cases.
\end{proof}

Finally, we note that $\Gamma, \Gamma_c \colon \catVB_M \to \catLCModt{C^\infty(M)}{\beta}$ (sections and compactly supported sections) are additive functors. 

Following \cite{capelle} we will now give the construction of the tensor product of locally convex modules. Let $A$ be a locally convex algebra and $\cH,\cK$ locally convex $A$-modules. Define $J_0$ as the sub-$\bZ$-module of $\cH \otimes \cK$ (the tensor product over $\bK$) generated by all elements of the form $ma \otimes n - m \otimes an$ with $a \in A$, $m \in \cH$ and $n \in \cK$. The vector spaces $\cH \otimes_A \cK$ and $(\cH \otimes \cK) / J_0$ are isomorphic (\cite[Theorem I.5.1, p.~9]{capelle}). Noting that the closure $\overline{J_0}$ again is a sub-$\bZ$-module of $\cH \otimes_\lambda \cK$, we have a locally convex space $\cH \otimes_{A,\lambda} \cK \coleq (\cH \otimes_\lambda \cK) / \overline{J_0}$. Denoting by $q \colon \cH \otimes_\lambda \cK \to (\cH \otimes_\lambda \cK) / \overline{J_0} $ the quotient mapping we obtain a bilinear mapping $\otimes_{A,\lambda} \coleq q \circ \otimes\colon \cH \times \cK \to \cH \otimes_{A,\lambda} \cK$.

We call $\cH \otimes_{A,\lambda} \cK$ the \emph{$\lambda$-tensor product} of $\cH$ and $\cK$ over $A$. It is a locally convex $A$-module, but in general its multiplication $m \colon A \times \cH \otimes_{A, \lambda} \cK \to \cH \otimes_{A, \lambda} \cK$ is only separately continuous. In fact, for given $a \in A$ we define the mapping $m_a \colon \cH \otimes_{A,\lambda} \cK \to \cH \otimes_{A, \lambda} \cK$ as the tensor product of the multiplication mapping $x \mapsto ax$ on $\cH$ and the identity on $\cK$, which both are continuous. We then set $m(a,z) \coleq m_a(z)$, which is continuous in $z$. For continuity in $a$, given $z = q(\sum x_i \otimes y_i) \in \cH \otimes_{A, \lambda} \cK$ with $x_i \in \cH$ and $y_i \in \cK$, the value $m(a,z)$ is given by $q(\sum a x_i \otimes y_i)$ which obviously is continuous in $a$.

The $\lambda$-tensor product of locally convex modules has the following universal property.

\begin{proposition}\label{univprop}
Let $A$ be a locally convex algebra and $\cH,\cK,\cM$ locally convex $A$-modules. Then given any $\lambda$-continuous $A$-bilinear mapping $f\colon \cH \times \cK \to \cM$ there exists a unique continuous $A$-linear mapping $g\colon \cH \otimes_{A,\lambda} \cK \to \cM$ such that $f = g \circ \otimes_{A,\lambda}$. Conversely, given any $g \in \cL_{A} ( \cH \otimes_{A, \lambda} \cK, \cM)$ the mapping $g \circ \otimes_{A,\lambda}$ is $\lambda$-continuous and $A$-bilinear from $\cH \times \cK$ to $\cM$.

This correspondence gives a vector space isomorphism between the space of all $\lambda$-continuous $A$-bilinear mappings $\cH \times \cK \to \cM$ and the space $\cL_A ( \cH \otimes_{A, \lambda}\cK, \cM)$.
\end{proposition}

\begin{proof}
To given $f$ as in the statement there corresponds a continuous mapping $\tilde f \colon \cH \otimes_\lambda \cK \to \cM$; moreover, $J_0 \subseteq \ker \tilde f$ and by continuity of $\tilde f$, also $\overline{J_0} \subseteq \ker \tilde f$. Hence, there exists a continuous mapping $g \colon \cH \otimes_{A, \lambda} \cK \to \cM$ such that $f = g \circ q \circ \otimes = g \circ \otimes_{A, \lambda}$. The converse is obvious.
\end{proof}

Let locally convex $A$-modules $\cH_i,\cK_i, \cM_i$ and $A$-linear continuous mappings $f_i \in \cL_A ( \cH_i, \cK_i )$ be given for $i=1,2$. With $p,q,r$ the respective quotient mappings, the continuous mapping $q \circ (f_1 \otimes f_2) \in \cL ( \cH_1 \otimes_\lambda \cH_2, \cK_1 \otimes_{A, \lambda} \cK_2)$ vanishes on the kernel of $p$, hence induces the continuous $A$-linear mapping $f_1 \otimes_{A, \lambda} f_2 \in \cL_A ( \cH_1 \otimes_{A, \lambda} \cH_2, \cK_1 \otimes_{A, \lambda} \cK_2 )$. For $g_i \in \cL_a ( \cK_i, \cM_i)$ ($i=1,2$) we have $(g_1 \circ g_2) \otimes_{A, \lambda} (f_1 \circ f_2) = (g_1 \otimes_{A, \lambda} g_2) \circ (f_1 \otimes_{A, \lambda} f_2)$, as is easily seen from the following diagram:
\[
 \xymatrix@C=1in{
 \cH_1 \otimes_{\lambda} \cH_2 \ar@/_4em/[dd]_{(g_1 \circ f_1) \otimes (g_2 \circ f_2)} \ar[d]^{f_1 \otimes f_2} \ar[r]^{p} & \cH_1 \otimes_{A, \lambda} \cH_2 \ar[d]_{f_1 \otimes_{A, \lambda} f_2} \ar@/^4em/[dd]^{(g_1 \circ g_2) \otimes_{A, \lambda} (f_1 \circ f_2)} \\
 \cK_1 \otimes_{\lambda} \cK_2 \ar[d]^{g_1 \otimes g_2} \ar[r]^{q} & \cK_1 \otimes_{A, \lambda} \cK_2 \ar[d]_{g_1 \otimes_{A, \lambda} g_2} \\
 \cM_1 \otimes_{\lambda} \cM_2 \ar[r]^{r} & \cM_1 \otimes_{A, \lambda} \cM_2
 }
\]
In fact we only need to use that $p$ is a quotient mapping and the whole diagram except for the part in question is commutes.

This makes $\otimes_{A, \lambda}$ a functor $\catLCMod A \times \catLCMod A \to \catLCMod A$ which obviously is additive.

Given a locally convex algebra $A$ and a locally convex $A$-module $\cH$ with $\lambda$-continuous multiplication, we have a canonical continuous $A$-linear mapping $A \otimes_{A, \lambda} \cH \to \cH$ whose inverse is given by $m \mapsto 1 \otimes_{A, \lambda} m$; this defines an isomorphism
\begin{equation}\label{atimesm}
A \otimes_{A, \lambda} \cH \cong \cH \textrm{ in }\catLCMod A.
\end{equation}

The following Lemma will be needed in Section \ref{distreg}.
\begin{lemma}\label{tonne}Let $A$ be a barrelled locally convex algebra. Then for locally convex $A$-modules $\cH,\cK$ at least one of which has hypocontinuous multiplication, $M \otimes_{A, \beta} N$ is a locally convex $A$-module with hypocontinuous multiplication.

In other words, we have an additive functor
\[ \otimes_{A,\beta} \colon \catLCModt{A}{\beta} \times \catLCMod{A} \to \catLCModt{A}{\beta}. \]
\end{lemma}
\begin{proof}
Suppose that $\cH$ has hypocontinuous multiplication $f \colon A \times \cH \to \cK$ (the case in which $\cK$ does is similar). $\cH \otimes_{A, \beta} \cK$ has separately continuous multiplication $m \colon A \times \cH \otimes_{A, \beta} \cK \to \cH \otimes_{A, \beta} \cK$ given by $m(a, x) \coleq (f_a \otimes_{A, \beta} \id)(x)$, where $f_a \coleq f(a,.) \in \cL(\cH, \cH)$. As $A$ is barrelled, $m$ is hypocontinuous with respect to bounded subsets of $\cH \otimes_{A, \beta} \cK$ (\cite[\S 40.2 (3) a), p.~158]{Koethe2}). It remains to show that for each bounded subset $A' \subseteq A$ the set 
\[ \{ f_a \otimes_{A,\beta} \id\ |\ a \in A' \} \subseteq \cL ( \cH \otimes_{A, \beta} \cK, \cH \otimes_{A, \beta} \cK ) \]
is equicontinuous. For this it suffices (\cite[p.~11]{zbMATH03145499}) to know that $\{ f_a\ |\ a \in A' \}$ is equicontinuous, which holds by assumption.
\end{proof}

\section{Reduction to trivial bundles}\label{sec_reduction}

We will now describe the general (classical) principle at work behind the proofs of isomorphisms \eqref{iso1} and \eqref{iso2}. Given a manifold $M$, consider two covariant functors $T,T'$ from $\catVB_M$ into any additive subcategory of $\catMod{C^\infty(M)}$. Suppose we have a natural transformation $\nu \colon T \to T'$, i.e., for each vector bundle $E$ there is a morphism $\nu_E \colon T(E) \to T'(E)$ such that for each vector bundle homomorphism $\mu \colon E \to E'$ covering the identity of $M$ the following diagram commutes:
\begin{equation}\label{twostar}
\begin{gathered}
 \xymatrix{
T(E) \ar[r]^{\nu_E} \ar[d]_{T(\mu)} & T'(E) \ar[d]^{T'(\mu)} \\
T(E') \ar[r]_{\nu_{E'}} & T'(E').
}
\end{gathered}
\end{equation}
Suppose we want to show that $\nu_E$ is a monomorphism, epimorphism or isomorphism for all $E$. We will show that this can easily be reduced to the case of the trivial line bundle $E = M \times \bK$ if $T$ and $T'$ are additive functors.

For this it is essential that for every vector bundle $E$ there exists a vector bundle $F$ such that $E \oplus F$ is trivial (\cite[Section 2.23, Theorem I, p.~76]{GHV}). Denoting by $\iota_E \colon E \to E \oplus F$ and $\pi_E \colon E \oplus F \to E$ the canonical injection and projection, respectively, we obtain the diagram
\[
 \xymatrix{
T(E) \ar[r]^{\nu_E} \ar[d]_{T(\iota_E)} & T'(E) \ar[d]^{T'(\iota_E)} \\
T(E \oplus F) \ar[r]^{\nu_{E \oplus F}} \ar[d]_{T(\pi_E)} & T'(E \oplus F) \ar[d]^{T'(\pi_E)} \\
T(E) \ar[r]^{\nu_E} & T'(E)
}
\]
which commutes because $\nu$ is natural. We see that if $\nu_{E \oplus F}$ is a monomorphism, epimorphism or isomorphism, $\nu_E$ has the same property; 
in the last case the inverse of $\nu_E$ is given by $T(\pi_E) \circ \nu_{E \oplus F}^{-1} \circ T'(\iota_E)$.

This way, the problem is reduced to the case where $E$ is a trivial vector bundle. But then it is of the form $(M \times \bK)^{(n)}$, i.e., the direct sum of $n$ copies of the trivial line bundle, and by the same reasoning as before the discussion is reduced to the case where $E = M \times \bK$. 

Hence, the main work lies in showing that $T$, $T'$ and $\nu$ have the desired properties, which is easy as soon as the respective categories are identified and a candidate for $\nu$ is found.

As an example, we have the isomorphism in $\catLCMod{C^\infty(M)}$
\begin{equation}\label{seciso}
\psi_{E,F} \colon \Gamma(M, E) \otimes_{C^\infty(M), \iota} \Gamma_c(M, F) \to \Gamma_c (M, E \otimes F)
\end{equation}
given by taking the fiberwise tensor product, i.e., $\psi(s \otimes t)(p) \coleq s(p) \otimes t(p)$. $T = \otimes_{C^\infty(M), \iota} \circ (\Gamma \times \Gamma_c)$ and $T' = \Gamma_c \circ \otimes$ are additive because $\Gamma$, $\Gamma_c$, $\otimes_{C^\infty(M), \iota}$ and $\otimes$ are, and $\psi$ is easily seen to be a natural transformation. By the above procedure the claim that $\psi_{E,F}$ is an isomorphism can be reduced to the case $E=F=M \times \bK$, which amounts to establishing $C^\infty(M) \otimes_{C^\infty(M), \iota} C^\infty_c(M) \cong C^\infty_c(M)$. This follows using \eqref{atimesm} from the fact that the module multiplication $C^\infty(M) \times C^\infty_c(M) \to C^\infty_c(M)$ is separately continuous and $\psi$ is an isomorphism. Note that the $\iota$-tensor product and the $\beta$-tensor product coincide here because $\Gamma(M, E)$ and $\Gamma_c(M, F)$ are barrelled.

\section{Representations of the space of distributions}

Let $\cH$ be a locally convex $C^\infty(M)$-module with multiplication
\begin{align*}
m & \colon C^\infty(M) \times \cH \to \cH \\
\intertext{and (via Proposition \ref{univprop}) associated linear continuous mapping}
\widetilde m & \colon C^\infty(M) \otimes_{C^\infty(M), \iota} \cH \to \cH.
\end{align*}
In order to obtain the second isomorphism of \eqref{iso1} (for which we set $\cH = \cD'(M)$) we endow $\Gamma(M, E) \otimes_{C^\infty(M)} \cH$ with the $\iota$-tensor product topology and $\cL_{C^\infty(M)} \bigl( \Gamma(M, E^*), \cH \bigr)$ with the simple topology.

We first show the following:
\begin{theorem}\label{thm_asdfasasdf}For any locally convex $C^\infty(M)$-module $\cH$, the following isomorphism of locally convex $C^\infty(M)$-modules holds:
\begin{equation}\label{claim}
\Gamma(M,E) \otimes_{C^\infty(M), \iota} \cH \cong \cL_{C^\infty(M),\sigma} \bigl( \Gamma(M, E^*), \cH\bigr).
\end{equation}
\end{theorem}

In order to apply the procedure of Section \ref{sec_reduction} we introduce some notation: let $E, E'$ be vector bundles over $M$ and $\mu \in \catVB_M(E, E')$; then $\Gamma(\mu) = \mu_* \in \cL_{C^\infty(M)} \bigl( \Gamma(M, E) , \Gamma(M, E') \bigr)$ denotes pushforward of sections and $(\Gamma \circ \textunderscore^*)(\mu) = \mu^* \in \cL_{C^\infty(M)} \bigl( \Gamma(M, (E')^*) , \Gamma(M, E^*) \bigr)$, which is defined via contraction by $(\mu^*s) \cdot t \coleq s \cdot (\mu_*t)$ for $s \in \Gamma(M, (E')^*)$ and $t \in \Gamma(M,E)$, is the pullback of sections of the dual bundle. The functors $T,T' \colon \catVB_M \to \catLCMod{C^\infty(M)}$ and the natural transformation $\nu \colon T \to T'$ are defined as follows: 
\begin{itemize}
 \item $T(E) \coleq \Gamma(M,E) \otimes_{C^\infty(M), \iota} \cH$, $T(\mu) \coleq \mu_* \otimes_{C^\infty(M), \iota} \id_\cH$. In other words, $T= (\textunderscore \otimes_{C^\infty(M), \iota} \cH) \circ \Gamma$.
 \item $T'(E) \coleq \cL_{C^\infty(M), \sigma} \bigl( \Gamma(M, E^*), \cH \bigr)$, $T'(\mu) \coleq [\ell \mapsto \ell \circ \mu^*]$. In other words, $T' = \cL_{C^\infty(M), \sigma} ( \textunderscore, \cH) \circ \Gamma \circ \textunderscore^*$.
 \item $\nu_E \in \cL_{C^\infty(M)} \bigl( T(E), T'(E) \bigr)$ is given by
 \[ \nu_E (x)(s) \coleq \widetilde m \bigl( ( c_s \otimes_{C^\infty(M), \iota} \id_\cH) (x) \bigr) \]
 for $x \in T(E)$ and $s \in \Gamma(M, E^*)$, where $c_s \in \cL_{C^\infty(M)} \bigl( \Gamma(M, E),  C^\infty(M) \bigr)$ denotes contraction with $s$. Given $x \in T(E)$, which can always be written in the form $x = q(\sum t_i \otimes h_i)$ with $t_i \in \Gamma(M,E)$, $h_i \in \cH$ and $q$ the quotient mapping, one has $\nu_E(x)(s) = \sum_i m(t_i \cdot s, h_i)$, from which continuity in $s$ and also in $x$ is obvious.
\end{itemize}

Because the functors $T$ and $T'$ are given by the composition of additive functors they also are additive. Next, we show that $\nu$ is a natural transformation: because for $s \in \Gamma (M, (E')^*)$ and $t \in \Gamma(M, E)$, $c_{\mu^* s}(t) = \mu^*s \cdot t = s \cdot \mu_*t = (c_s \circ \mu_*)(t)$ we have $\bigl(T'(\mu) \circ \nu_E\bigr)(x)(s) = \nu_E(x)(\mu^*s) = \widetilde m \bigl( (c_{\mu^*s} \otimes_{C^\infty(M), \iota} \id_\cH)(x)\bigr) = \widetilde m \bigl(((c_s \circ \mu_*) \otimes_{C^\infty(M), \iota} \otimes \id_\cH)(x)\bigr) = \nu_{E'} \bigl( (\mu_* \otimes_{C^\infty(M), \iota} \id_\cH)(x)\bigr)(s) = \bigl(\nu_{E'} \circ T(\mu)\bigr)(x)(s)$.

Hence, \eqref{claim} follows because by Lemma \ref{blahlinprop} and \eqref{atimesm}
\[ \nu_{M \times \bK} \colon C^\infty(M) \otimes_{C^\infty(M), \iota} \cH \to \cL_{C^\infty(M), \sigma} \bigl( C^\infty(M), \cH \bigr) \]
is an isomorphism with inverse $l \mapsto 1 \otimes_{C^\infty(M), \iota} l(1)$ in $\catLCMod{C^\infty(M)}$.

In order to prove \eqref{iso1} we will show the outer spaces to be isomorphic, which means that
\[ \cD'(M, E) \cong \cL_{C^\infty(M), \sigma} \bigl( \Gamma(M, E^*), \cD'(M) \bigr) \]
in $\catLCMod{C^\infty(M)}$. Once more, we reduce everything to the case of trivial bundles. Let
\[ \psi_{E^*} \colon \Gamma(M, E^*) \otimes_{C^\infty(M), \iota} \Gamma_c \bigl(M, \Vol(M) \bigr) \to \Gamma_c \bigl( M, E^* \otimes \Vol(M) \bigr) \]
denote isomorphism \eqref{seciso}. With $\mu$ as before we define our functors $T$ and $T' \colon \catVB_M \to \catLCMod{C^\infty(M)}$ and the natural transformation $\nu \colon T \to T'$ as follows:
\begin{itemize}
 \item $T(E) \coleq \cD'(M, E)$, $T(\mu)(u) \coleq u \circ \psi_{E^*} \circ (\mu^* \otimes_{C^\infty(M), \iota} \id) \circ \psi_{(E')^*}^{-1} \in \cD'(M, E')$ for $u \in \cD'(M, E)$.
 \item $T'(E) \coleq \cL_{C^\infty(M), \sigma} \bigl( \Gamma(M, E^*), \cD'(M) \bigr)$ and $T'(\mu)(\ell) \coleq \ell \circ \mu^*$ for each $\ell \in \cL_{C^\infty(M)} \bigl( \Gamma(M, E^*), \cD'(M) \bigr)$.
 \item $\nu_E (u)(s)(\omega) \coleq \langle u, \psi_{E^*} ( s \otimes_{C^\infty(M), \iota} \omega) \rangle$ for $u \in \cD'(M, E)$, $s \in \Gamma(M, E^*)$ and $\omega \in \Gamma_c\bigl(M, \Vol(M)\bigr)$. That $\nu_E(u)(s)$ is in $\cD'(M)$ and $\nu_E(u)$ is $C^\infty(M)$-linear is clear. Finally, $\nu_E(u)$ is continuous: in fact, let $s_n \to 0$ in $\Gamma(M, E^*)$. Then because $\Gamma_c\bigl(M, \Vol(M) \bigr)$ is barrelled, we have $[\omega \mapsto \langle u, \psi_{E^*} (s_n \otimes_{C^\infty(M), \iota} \omega ) \rangle ]\to 0$ in $\cD'(M)$ if it converges pointwise, which obviously is the case.
\end{itemize}
As before, $T$ and $T'$ are additive because they are given by the composition of additive functors; one easily verifies that $\nu$ is a natural transformation. 
Hence, the claim is reduced to the trivial line bundle $M \times \bK$, for which it reads
\begin{equation}\label{alpha}
\cD'(M) \cong \cL_{C^\infty(M), \sigma} \bigl( C^\infty(M), \cD'(M) \bigr).
\end{equation}

By Lemma \ref{blahlinprop} this is an isomorphism in $\catLCMod{C^\infty(M)}$ with inverse $\ell \mapsto \ell(1)$. This completes the proof of \eqref{iso1}.

In order to round off this first result we show that we can in fact use stronger topologies:
\begin{align}
 \Gamma(M, E) \otimes_{C^\infty(M), \iota } \cD'(M) &= \Gamma (M, E) \otimes_{C^\infty(M), \beta} \cD'(M), \nonumber  \\
\cL_{C^\infty(M), \sigma} \bigl( \Gamma(M, E^*), \cD'(M) \bigr) &= \cL_{C^\infty(M), \beta} \bigl( \Gamma(M, E^*), \cD'(M) \bigr). \label{betaa}
\end{align}
For the first, we simply note that $\Gamma(M, E)$ and $\cD'(M)$ are barrelled, hence every seperately continuous bilinear mapping from $\Gamma(M, E) \times \cD'(M)$ into any locally convex space is hypocontinuous (\cite[\S 40 2.(5) a), p.~159]{Koethe2}).

For the second, we need to show that every $0$-neighborhood of the $\beta$-topology on $\cL_{C^\infty(M)} \bigl ( \Gamma(M, E^*), \cD'(M) \bigr)$, which can be taken to be of the form $U_{B, V} \coleq \{ \ell : \ell(B) \subseteq V \}$ for $B \subseteq \Gamma(M, E^*)$ bounded and $V \subseteq \cD'(M)$ a $0$-neighborhood, contains a $0$-neighborhood of the $\sigma$-topology, i.e., one of the form $U_{B', V'}$ where $B' \subseteq \Gamma(M, E^*)$ is finite and $V'$ again is a $0$-neighborhood in $\cD'(M)$.

For this we need to use the fact that $\Gamma(M, E^*)$ is a finitely generated $C^\infty(M)$-module as follows: let $F$ be such that $E^* \oplus F$ is trivial and choose a basis $(b_i)_{i=1 \dotsc n}$ of the $C^\infty(M)$-module $\Gamma(E^* \oplus F)$ with corresponding dual basis $(\beta^i)_{i=1 \dotsc m}$. Denoting by $\iota$ and $\pi$ the canonical injection of $E^*$ into $E^* \oplus F$ and the corresponding projection, respectively, we can write any $s \in \Gamma(M, E^*)$ as
\[
 s = \pi_* ( \iota_* s) = \pi_* \bigl( \sum_{i=1}^n \beta^i ( \iota_* s ) \cdot b_i ) = \sum_{i=1}^n \beta^i ( \iota_* s ) \cdot \pi_* b_i.
\]
Hence, $B' \coleq ( \pi_* b_i)_{i=1 \dotsc n}$ is a generating set of $\Gamma(M, E^*)$. Because $\beta^i \circ \iota_*$ is a continuous linear mapping $\Gamma(M, E^*) \to C^\infty(M)$ the set $D \coleq \{ s^i\ |\ s \in B,\ i = 1\dotsc n \} \subseteq C^\infty(M)$ with with $s^i \coleq \beta^i ( \iota_* s )$ for $s \in B$ is bounded.

Choose $0$-neighborhoods $V', V''$ in $\cD'(M)$ such that $V'' + \dotsc + V''$ ($n$ summands) is contained in $V$ and $D \cdot V' \subseteq V''$. Then $U_{B', V'}$ is a $0$-neighborhood for the topology of simple convergence and for $s \in B$ and $\ell \in U_{B', V'}$ we have
\[ \ell(s) = \ell ( \sum_i s^i \pi_* b_i ) = \sum_i s^i \ell(\pi_* b_i) \in \sum_i D \cdot V' \subseteq \sum_i V'' \subseteq V \]
which means that $U_{B', V'} \subseteq U_{B, V}$. Hence, the topologies of simple and bounded convergence coincide.

\section{Regularization of vector valued distributions}\label{distreg}

In this section we are going to show isomorphism \eqref{iso2}, for which we require some preliminaries. First of all, we recall L.~Schwartz' notion of $\e$-product from \cite{zbMATH03145498}. For any locally convex space $\cH$, let $\cH'_c$ denote the dual space of $\cH$ endowed with the topology of uniform convergence on absolutely convex 
compact subsets of $\cH$. The $\e$-product $\cH \eps \cK$ of two locally convex spaces $\cH$ and $\cK$ is defined as the vector space of all bilinear mappings $\cH'_c \times \cK'_c \to \bK$ which are hypocontinuous with respect to equicontinuous subsets of $\cH'$ and $\cK'$. It is endowed with the topology of uniform convergence on products of equicontinuous subsets of $\cH'$ and $\cK'$ (\cite[\S 1, Definition, p.~18]{zbMATH03145498}). There is a canonical isomorphism $\cH \eps \cK \cong \cL_\e ( \cH'_c, \cK)$, where the latter space is endowed with the topology of uniform convergence on equicontinuous subsets of $\cH'$ (\cite[\S 1, Corollaire 2 to Proposition 4, p.~34]{zbMATH03145498}). By \cite[\S 1, Proposition 3, p.~29]{zbMATH03145498} $\cH \eps \cK$ is complete if $\cH$ and $\cK$ are complete. Noting that $\cH \otimes \cK$ is canonically contained in $\cH \eps \cK$ (\cite[p.~19]{zbMATH03145498}), given locally convex spaces $\cH_i$, $\cK_i$ for $i=1,2$ and linear continuous mappings $f \colon \cH_1 \to \cH_2$ and $g \colon \cK_1 \to \cK_2$ there is a canonical continuous linear map $f \eps g \colon \cH_1 \eps \cK_1 \to \cH_2 \eps \cK_2$ extending the map $f \otimes g \colon \cH_1 \otimes \cK_1 \to \cH_2 \otimes \cK_2$ (\cite[\S1 Proposition 1, p.~20]{zbMATH03145498}).

We will need the $\e$-product in two ways. First, fix two vector bundles $E \to M$ and $F \to N$. Viewing $\Gamma(M, E) \otimes \Gamma(N, F)$ (endowed with the projective tensor topology) as a dense subspace of $\Gamma(M \times N, E \boxtimes F)$, 
its completion $\Gamma(M, E) \widehat\otimes \Gamma(N, F)$ is isomorphic as a locally convex space to $\Gamma(M, E) \eps \Gamma(N, F)$ (\cite[\S 1, Corollaire 1 to Proposition 11, p.~47]{zbMATH03145498}) because $\Gamma(M,E)$ and $\Gamma(N, F)$ are complete and have the approximation property. Second, we note that on $\cD'(M, E)$ the topology of bounded convergence coincides with the topology of absolutely convex compact convergence.
Because $\Gamma_c\bigl(M, E^* \otimes \Vol(M) \bigr)$ is barrelled, on $\cL(\cD'(M, E), \Gamma(N,F))$ the topologies of bounded and equicontinuous convergence coincide as a consequence of the Banach-Steinhaus theorem. Summarizing, we have
\begin{align*}
 \Gamma(M, E) \widehat\otimes \Gamma(N, F) & = \Gamma(M \times N, E \boxtimes F) \cong \Gamma(M, E) \eps \Gamma(N, F), \\
 \cL_\beta \bigl( \cD'(M, E), \Gamma(N, F)\bigr) & = \Gamma_c\bigl(M, E^* \otimes \Vol(M)\bigr) \eps \Gamma(N, F), \\
 \cL_\beta ( \cD'(M), C^\infty(N) ) &= \Gamma_c(M, \Vol(M)) \eps C^\infty(N).
\end{align*}
In order to define a locally convex $C^\infty(M \times N)$-module structure on these spaces and establish the desired isomorphism \eqref{iso2} we will employ the following Lemma. While its proof can be based on L.~Schwartz' \textit{Th\'eor\`eme de croisement} (\cite[\S 2, Proposition 2, p.~18]{zbMATH03145499}; cf.~also \cite[Proposition 2]{zbMATH06260723}) we prefer to give a direct proof.

\begin{lemma}\label{innsbruck}Let $A$ and $B$ be nuclear Fr\'echet spaces and $M_1,M_2, N_1, N_2$ locally convex spaces with $M_2$ and $N_2$ complete. Suppose we are given two hypocontinuous bilinear mappings $f \colon A \times M_1 \to M_2$ and $g \colon B \times N_1 \to N_2$.
Then there exists a unique separately continuous bilinear mapping
 \[ \sigma \colon A \widehat \otimes B \times M_1 \eps N_1 \to M_2 \eps N_2 \]
satisfying
\begin{equation}\label{nikolaus}
\sigma ( a \otimes b, u) \coleq \bigl( f(a, \cdot) \eps g(b, \cdot)\bigr)(u)\qquad \forall a \in A, b \in B, u \in M_1 \eps M_2.
\end{equation}
Moreover, $\sigma$ even is hypocontinuous.
\end{lemma}
\begin{proof}
 For $a \in A$ and $b \in B$ we set $f_a \coleq f(a, \cdot) \in \cL(M_1, M_2)$ and $g_b \coleq g(b, \cdot) \in \cL(N_1, N_2)$. We define a trilinear map $\tilde \sigma \colon A \times B \times M_1 \eps N_1 \to M_2 \eps N_2$ by $\tilde \sigma(a, b, u) \coleq (f_a \eps g_b)(u)$. In order to show that it is hypocontinuous fix bounded subsets $A' \subseteq A$, $B' \subseteq B$ and $D \subseteq M_1 \eps N_1$ and a $0$-neighborhood $U$ in $M_2 \eps N_2$ of the form $U = (X \times Y)^\circ$, where $X \subseteq M_2'$ and $Y \subseteq N_2'$ are arbitrary equicontinuous subsets.
 
(i) As $\{f_a\ |\ a \in A'\}$ and $\{g_b\ |\ b \in B'\}$ are equicontinuous by assumption, $X' \coleq \{ m_2' \circ f_a\ |\ m_2' \in X, a \in A'\} \subseteq M_1'$ and $Y' \coleq \{ n_2' \circ g_b\ |\ n_2' \in Y, b \in B'\} \subseteq N_1'$ are equicontinuous as well and $V \coleq (X' \times Y')^\circ$ is a $0$-neighborhood in $M_1 \eps N_1$ such that $\tilde \sigma(A', B', V) \subseteq U$.

(ii) We need to find a $0$-neighborhood $V$ in $A$ such that $\tilde \sigma(V, B', D) \subseteq U$. By \cite[\S 1, Proposition 2 bis, p.~28]{zbMATH03145498} $D$ is an $\e$-equihypocontinuous subset of $\cL\bigl( (M_1)'_\beta \times (N_1)'_\beta, \bK \bigr)$. This means that there exists a $0$-neighborhood of the form $C^\circ$ in $(M_1')_b$, with $C \subseteq M_1$ bounded, such that $u(C^\circ, Y') \in \bD$ for all $u \in D$, where $\bD$ is the closed unit disk of $\bK$. Hence, we only need to choose $V$ such that $m_2' \circ f(a,c) \in \bD$ for all $m_2' \in X$, $a \in V$ and $c \in C$, which is possible by the assumption on $f$. This means that $m_2' \circ f(a, \cdot) \in C^\circ$ and hence $\tilde \sigma(V, B', D)(X,Y) \subseteq D(X \circ f(V, \cdot), Y') \in \bD$.

(iii) Similarly, one can find a $0$-neighborhood $V \subseteq B$ such that $\tilde \sigma(A', V, D) \subseteq U$.

This shows that $\tilde \sigma$ is hypocontinuous as claimed. Because $M_2 \eps N_2$ is complete (\cite[\S 1, Proposition 3, p.~29]{zbMATH03145498}), for each $u$ the (by \cite[\S 40 2.(1) a), p.~158]{Koethe2}) continuous map $\tilde \sigma_u \coleq \tilde \sigma(., ., u)$ has a unique extension to a linear continuous map $\sigma_u \colon A \widehat\otimes B \to M_2 \eps N_2$. We now define $\sigma \colon A \widehat \otimes B \times M_1 \eps N_1 \to M_2 \eps N_2$ by $\sigma(x,u) \coleq \sigma_u(x)$, which by definition is linear and continuous in $x$. For linearity in $u$, it suffices by continuity in $x$ to verify $\sigma(x, u_1 + \lambda u_2) = \sigma(x, u_1) + \lambda \sigma(x, u_2)$ for $u_1,u_2 \in M_1 \eps M_2$ and $\lambda \in \bK$ for $x$ in the dense subspace $A \otimes B$, where it is evident. For hypocontinuity let $Z \subseteq A \widehat \otimes B$ be bounded and $D,U$ as above. Then by \cite[Theorem 21.5.8, p.~495]{Jarchow}) there exist bounded subsets $A' \subseteq A$ and $B' \subseteq B$ such that $Z \subseteq \cacx(A' \otimes B')$, where $\cacx$ denotes the absolutely convex closed hull of a set. Choosing $V$ as in (i) above we then have $\sigma(Z, V) \subseteq \cacx (\sigma(A' \otimes B', V)) = \cacx (\tilde \sigma(A', B', V)) \subseteq \cacx U = U$ because $\sigma$ is continuous in the first variable and $U$ is closed and absolutely convex. On the other hand, $\{ \sigma(., u)\ |\ u \in D\}$ is equicontinuous because $\{ \tilde \sigma (.,.,u)\ |\ u \in D \}$ is separately equicontinuous (\cite[\S 40 2.(2), p.~158]{Koethe2}). Because $A \otimes B$ is dense in $A \widehat\otimes B$, \eqref{nikolaus} uniquely defines $\sigma$.
\end{proof}

\begin{corollary}\begin{enumerate}[(i)] Let $A$ and $B$ be nuclear Fr\'echet locally convex algebras.
\item The map $A \otimes B \times A \otimes B \to A \otimes B$, $(a_1 \otimes b_1, a_2 \otimes b_2) \mapsto a_1 a_2 \otimes b_1b_2$ extends uniquely to a continuous bilinear map $A \widehat\otimes B \times A \widehat\otimes B \to A \widehat\otimes B$, turning $A \widehat\otimes B$ into a locally convex algebra.
\item Given a complete locally convex $A$-module $M$ and a complete locally convex $B$-module $N$ such that the multiplications $f$ of $M$ and $g$ of $N$ are hypocontinuous, the map $A \otimes B \times M \eps N \to M \eps N$, $(a \otimes b, u) \mapsto (f_a \eps g_b)(u)$ extends uniquely to a hypocontinuous bilinear mapping $A \widehat \otimes B \times M \eps N \to M \eps N$ turning $M \eps N$ into a locally convex $A \widehat\otimes B$-module with hypocontinuous multiplication.
\item In particular, $\eps$ is an additive functor
\[ \catLCModt{A}{\beta} \times \catLCModt{B}{\beta} \to \catLCModt{A \widehat\otimes B}{\beta}. \]
\end{enumerate}
\end{corollary}
\begin{proof}Lemma \ref{innsbruck} gives the necessary mappings. The axioms for the algebra and module structure are easily verified by restricting to the dense subspace $A \otimes B$ of $A \widehat\otimes B$.
\end{proof}

In order to apply the procedure of Section \ref{sec_reduction} we define functors $T,T' \colon \catVB_M \times \catVB_N \to \catLCModt{C^\infty(M \times N)}{\beta}$ by
\begin{align*}
 T & \coleq \eps \circ (\Gamma_c \times \Gamma) \circ \bigl( ( \textunderscore \otimes \Vol(M)) \times \id \bigr) \circ (\textunderscore^* \times \id), \\
 T' &\coleq \cL_{C^\infty(M \times N), \beta} \bigl( \textunderscore, \Gamma_c(M, \Vol(M)) \eps C^\infty(N)\bigr) \circ \widehat\otimes \circ (\Gamma \times \Gamma) \circ ( \id \times \textunderscore^*)
\end{align*}
Because they are compositions of additive functors, $T$ and $T'$ are additive. In order to define $\nu_{E,F} \colon T(E,F) \to T'(E,F)$ we apply Lemma \ref{innsbruck} to $A = \Gamma(M, E)$, $B = \Gamma(N, F^*)$, $M_1 = \Gamma_c(M, E^* \otimes \Vol(M))$, $N_1 = \Gamma(N, F)$, $M_2 = \Gamma_c(M, \Vol(M))$, $N_2 = C^\infty(N)$, $g \colon \Gamma(N, F^*) \times \Gamma(N, F) \to C^\infty(N)$ given by contraction and $f$ defined via
\[
 \xymatrix@C=7em{
 \Gamma(M, E) \times \Gamma_c\bigl(M, E^* \otimes \Vol(M)\bigr) \ar[r]^-f \ar[d]_{\id \times \psi^{-1}} & \Gamma_c\bigl(M, \Vol(M)\bigr) \\
 \Gamma(M, E) \times \Gamma(M, E^*) \widehat\otimes \Gamma_c\bigl(M, \Vol(M)\bigr) \ar[r]_-{(t,s) \mapsto (c_t \otimes \id)(s)} & C^\infty(M) \widehat\otimes \Gamma_c\bigl(M, \Vol(M)\bigr) \ar[u]_M
 }
\]
where $c_t \colon \Gamma(M, E^*) \to C^\infty(M)$ is contraction with $t \in \Gamma(M, E)$, which is a $C^\infty(M)$-linear continuous map, $\psi$ is the isomorphism \eqref{seciso} and
\[ M \colon C^\infty(M) \widehat\otimes \Gamma_c\bigl(M, \Vol(M)\bigr) \to \Gamma_c\bigl(M, \Vol(M)\bigr)\]
denotes the linear continous mapping canonically associated to the module multiplication of $\Gamma_c\bigl(M, \Vol(M)\bigr)$. Explicitly, $f$ is given by the map $f(t,s) \coleq M\bigl((c_t \otimes \id)(\psi^{-1}(s))\bigr)$ and obviously is hypocontinuous. This way, Lemma \ref{innsbruck} defines a mapping
\[ \sigma \colon \Gamma(M \times N, E \boxtimes F^*) \times \Gamma_c\bigl(M, E^* \otimes \Vol(M)\bigr) \eps \Gamma(N, F) \to \Gamma_c\bigl(M, \Vol(M)\bigr) \eps C^\infty(N) \]
which is hypocontinuous with respect to bounded subsets of $\Gamma(M \times N, E \boxtimes F^*)$, thus by \cite[\S 40.1 (3) a), p.~156]{Koethe2} induces a linear continuous mapping
\begin{multline*}
 \nu_{E,F} \colon \Gamma_c\bigl(M, E^* \otimes \Vol(M)\bigr) \eps \Gamma(N, F)\\
\to \cL_\beta \bigl( \Gamma(M \times N, E \boxtimes F^*), \Gamma_c(M, \Vol(M)) \eps C^\infty(N)\bigr)
\end{multline*}
by setting $(\nu u)(s) \coleq \sigma(s, u)$. It is easily verified that for $a \otimes b \in C^\infty(M) \otimes C^\infty(N)$ and $s \otimes t \in \Gamma(M, E) \otimes \Gamma(N, F^*)$, $(\nu u)((a \otimes b) \cdot (s \otimes t)) = (a \otimes b) \cdot ((\nu u)(s \otimes t))$ which implies that $\nu u$ is $C^\infty(M \times N$)-linear. The equality $\nu(h \cdot u)(v) = h \cdot \nu(u)(v)$ holds for all $h = a \otimes b$ and $v = s \otimes t$ as above, and hence also for $h \in C^\infty(M \times N)$ and $v \in \Gamma(M \times N, E \boxtimes F^*)$ because both sides are continuous in $h$ and $v$ separately; this means that $\nu$ itself is $C^\infty(M \times N)$-linear.

Summarizing, for each pair $E,F$ we have defined a $C^\infty(M \times N)$-linear continuous map $\nu_{E,F} \colon T(E,F) \to T'(E,F)$. Moreover, one easily verifies that $\nu$ is a natural transformation from $T$ to $T'$. It follows that in order for $\nu$ to be a natural isomorphism, it suffices to verify this for the case of trivial line bundles $E = M \times \bK$ and $F = N \times \bK$; but in this case one can immediately write down the inverse of $\nu$, namely $\nu^{-1} \ell \coleq \ell (1)$ with $1 \in C^\infty(M \times N)$. Together with Theorem \ref{thm_asdfasasdf} and \eqref{betaa} this establishes \eqref{iso2}.

\section{Conclusion}

The topological variant of isomorphisms \eqref{iso1} and \eqref{iso2} hence reads as follows:
\begin{gather*}
\cD'(M, E) \cong \Gamma(M, E) \otimes_{C^\infty(M), \beta} \cD'(M) \cong \cL_{C^\infty(M), \beta} \bigl( \Gamma(M, E^*), \cD'(M) \bigr) \\
\begin{aligned}
\cL_\beta \bigl( \cD'(M, E), \Gamma(N, F) \bigr) & \cong \Gamma ( M \times N, E^* \boxtimes F ) \otimes_{C^\infty(M \times N), \beta} \cL_\beta \bigl( \cD'(M), C^\infty(N) \bigr) \\
&\cong \cL_{C^\infty(M \times N), \beta} \bigl( \Gamma(M \times N, E \boxtimes F^*), \cL_\beta ( \cD'(M), C^\infty(N) ) \bigr).
\end{aligned}
\end{gather*}

\textbf{Acknowledgements. } This work was supported by the Austrian Science Fund (FWF) grant P26859-N25.


\begin{thebibliography}{10}

\bibitem{zbMATH06260723}
C.~{Bargetz} and N.~{Ortner}.
\newblock {Convolution of vector-valued distributions: a survey and
  comparison.}
\newblock {\em {Diss. Math.}}, 495:51, 2013.

\bibitem{zbMATH00626734}
F.~{Borceux}.
\newblock {\em {Handbook of categorical algebra. Volume 1: Basic category
  theory.}}
\newblock Cambridge: Cambridge Univ. Press, 1994.

\bibitem{zbMATH05273392}
F.~{Borceux}.
\newblock {\em {Handbook of categorical algebra. 2: Categories and
  structures.}}
\newblock Cambridge: Cambridge University Press, 2008.

\bibitem{capelle}
J.~Capelle.
\newblock {\em {Convolution on homogeneous spaces}}.
\newblock PhD thesis, University of Groningen, 1996.

\bibitem{zbMATH03423321}
J.~Dieudonn{\'e}.
\newblock {\em Treatise on analysis. {V}ol. {III}}.
\newblock Academic Press, New York-London, 1972.
\newblock Pure and Applied Mathematics, Vol. 10-III.

\bibitem{GHV}
W.~Greub, S.~Halperin, and R.~Vanstone.
\newblock {\em {Connections, curvature, and cohomology. {V}ol. {I}: {D}e {R}ham
  cohomology of manifolds and vector bundles}}, volume~47 of {\em {Pure and
  Applied Mathematics}}.
\newblock Academic Press, New York, 1972.

\bibitem{found}
M.~{Grosser}, E.~{Farkas}, M.~{Kunzinger}, and R.~{Steinbauer}.
\newblock {On the foundations of nonlinear generalized functions. I, II.}
\newblock {\em {Mem. Am. Math. Soc.}}, 729, 2001.

\bibitem{GKOS}
M.~Grosser, M.~Kunzinger, M.~Oberguggenberger, and R.~Steinbauer.
\newblock {\em {Geometric theory of generalized functions with applications to
  general relativity}}, volume 537 of {\em {Mathematics and its Applications}}.
\newblock Kluwer Academic Publishers, Dordrecht, 2001.

\bibitem{global}
M.~Grosser, M.~Kunzinger, R.~Steinbauer, and J.~A. Vickers.
\newblock {A Global Theory of Algebras of Generalized Functions.}
\newblock {\em Adv. Math.}, 166(1):50--72, 2002.

\bibitem{global2}
M.~Grosser, M.~Kunzinger, R.~Steinbauer, and J.~A. Vickers.
\newblock {A global theory of algebras of generalized functions. II. Tensor
  distributions.}
\newblock {\em New York J. Math.}, 18:139--199, 2012.

\bibitem{zbMATH03199982}
A.~{Grothendieck}.
\newblock {Produits tensoriels topologiques et espaces nucl\'eaires.}
\newblock {\em {Mem. Am. Math. Soc.}}, 16, 1955.

\bibitem{MR1093462}
A.~Y. Helemskii.
\newblock {\em The homology of {B}anach and topological algebras}, volume~41 of
  {\em Mathematics and its Applications (Soviet Series)}.
\newblock Kluwer Academic Publishers Group, Dordrecht, 1989.

\bibitem{Jarchow}
H.~Jarchow.
\newblock {\em {Locally Convex Spaces}}.
\newblock B. G. Teubner, Stuttgart, 1981.

\bibitem{Koethe2}
G.~K{\"o}the.
\newblock {\em {Topological vector spaces. {II}}}, volume 237 of {\em
  {Grundlehren der Mathematischen Wissenschaften [Fundamental Principles of
  Mathematical Science]}}.
\newblock Springer-Verlag, New York, 1979.

\bibitem{foundgeom}
M.~Kunzinger and R.~Steinbauer.
\newblock {Foundations of a nonlinear distributional geometry.}
\newblock {\em Acta Appl. Math.}, 71(2):179--206, 2002.

\bibitem{genpseudo}
M.~Kunzinger and R.~Steinbauer.
\newblock {Generalized pseudo-{R}iemannian geometry.}
\newblock {\em Trans. Am. Math. Soc.}, 354(10):4179--4199, 2002.

\bibitem{Lee}
J.~M. Lee.
\newblock {\em {Introduction to Smooth Manifolds}}.
\newblock Springer, New York, 2013.

\bibitem{MR857807}
A.~Mallios.
\newblock {\em Topological algebras. {S}elected topics}, volume 124 of {\em
  North-Holland Mathematics Studies}.
\newblock North-Holland Publishing Co., Amsterdam, 1986.
\newblock Notas de Matem{\'a}tica [Mathematical Notes], 109.

\bibitem{gdnew}
E.~A. Nigsch.
\newblock {A new approach to diffeomorphism invariant algebras of generalized
  functions.}
\newblock {\em Proc. Edinb. Math. Soc.}, 2013.
\newblock To appear.

\bibitem{sectop}
E.~A. {Nigsch}.
\newblock {Bornologically isomorphic representations of distributions on
  manifolds.}
\newblock {\em {Monatsh. Math.}}, 170(1):49--63, 2013.

\bibitem{papernew}
E.~A. Nigsch.
\newblock {The functional analytic foundation of Colombeau algebras.}
\newblock {\em J. Math. Anal. Appl.}, 421(1):415--435, 2015.

\bibitem{Schaefer}
H.~H. Schaefer.
\newblock {\em {Topological Vector Spaces}}.
\newblock Springer--Verlag, New York, 1971.

\bibitem{Schwartz}
L.~Schwartz.
\newblock {Sur l'impossibilit{\'e} de la multiplication des distributions.}
\newblock {\em Comptes Rendus de L'Acad{\'e}mie des Sciences}, 239:847--848,
  1954.

\bibitem{FDVV}
L.~Schwartz.
\newblock {Espaces de fonctions diff{\'e}rentiables {\`a} valeurs
  vectorielles.}
\newblock {\em J. Anal. Math.}, 4, 1955.

\bibitem{zbMATH03145498}
L.~{Schwartz}.
\newblock {Th{\'e}orie des distributions {\`a} valeurs vectorielles.}
\newblock {\em {Ann. Inst. Fourier}}, 7, 1957.

\bibitem{zbMATH03145499}
L.~{Schwartz}.
\newblock {Th{\'e}orie des distributions {\`a} valeurs vectorielles. II.}
\newblock {\em {Ann. Inst. Fourier}}, 8, 1958.

\bibitem{genrel}
R.~Steinbauer and J.~A. Vickers.
\newblock {The use of generalized functions and distributions in general
  relativity}.
\newblock {\em Classical Quantum Gravity}, 23(10):R91--R114, 2006.

\bibitem{zbMATH03426281}
J.~L. {Taylor}.
\newblock {Homology and cohomology for topological algebras.}
\newblock {\em {Adv. Math.}}, 9:137--182, 1972.

\end{thebibliography}
\end{document}